\documentclass[11pt]{amsart}
\usepackage{amssymb}
\usepackage{verbatim}
\usepackage{amsmath}
\usepackage{bm}
\usepackage{amssymb,latexsym}
\usepackage{enumerate}

\newcommand{\dbar}{\ensuremath{\overline\partial}}

\newcommand{\C}{\ensuremath{\mathbb{C}}}

\makeatletter
\newcommand{\sumprime}{\if@display\sideset{}{'}\sum%
            \else\sum'\fi}
\makeatother

\begin{document}

\numberwithin{equation}{section}

\newtheorem{theorem}{Theorem}[section]
\newtheorem{proposition}[theorem]{Proposition}
\newtheorem{conjecture}[theorem]{Conjecture}
\def\theconjecture{\unskip}
\newtheorem{corollary}[theorem]{Corollary}
\newtheorem{lemma}[theorem]{Lemma}
\newtheorem{observation}[theorem]{Observation}
\newtheorem{definition}{Definition}
\numberwithin{definition}{section} 
\newtheorem{remark}{Remark}
\def\theremark{\unskip}
\newtheorem{kl}{Key Lemma}
\def\thekl{\unskip}
\newtheorem{question}{Question}
\def\thequestion{\unskip}
\newtheorem{example}{Example}
\def\theexample{\unskip}
\newtheorem{problem}{Problem}


\address{DEPARTMENT OF MATHEMATICAL SCIENCES, NORWEGIAN UNIVERSITY OF SCIENCE AND TECHNOLOGY, NO-7491 TRONDHEIM, NORWAY}
\email{xu.wang@ntnu.no}

\title[Hodge theory]{Notes on variation of Lefschetz star operator and $T$-Hodge theory}

 \author{Xu Wang}
\date{\today}

\begin{abstract} These notes were written to serve as an easy reference for \cite{Wang-AF}. All the results in this presentation are well-known (or quasi-well-known) theorems in Hodge theory. Our main purpose was to give a unified approach based on a variation formula of  the Lefschetz star operator, following \cite{Wang-k}. It fits quite well with Timorin's $T$-Hodge theory, i.e. the Hodge theory on the space of differential forms divided by $T$ (i.e. forms like $T\wedge u$), where $T$ is a finite wedge product of K\"ahler forms.
\bigskip



\end{abstract}
\maketitle



\tableofcontents

\section{Preliminaries}

\subsection{Primitivity: linear setting} Let $V$ be an $N$-dimensional real vector space. Let $\omega$ be a bilinear form on $V$. We call $\omega$ a \emph{symplectic form} if $\omega$ is non-degenerate and $\omega\in \wedge^2V^*$, 
i.e. $\omega(u,v)=-\omega(v,u), \ \forall  \ u, v\in V$.

\begin{proposition} Assume that there is a symplectic form $\omega$ on $V$. Then $N=2n$ for some integer $n$ and there exists a base, say $\{e^*_1, f^*_1; \cdots; e^*_n, f^*_n\}$, of $V^*$ such that 
$$
\omega=\sum_{j=1}^n e^*_j\wedge f^*_j.
$$
\end{proposition}

\begin{proof} Since $\omega$ is non-degenerate, we know that $N\geq 2$. If $N=2$ and $\omega(e,f)=1$ then
$$
\omega=e^*\wedge f^*, 
$$
where $\{e^*, f^*\}$ is the dual base of $\{e, f\}$. Assume that $N\geq 3$, consider
$$
V':=\{u\in V: \omega(u, e)=\omega(u, f)=0\}.
$$
Then for every $u\in V$, we have
$$
u':=u-\omega(u,f) e+\omega(u,e)f\in V',
$$
and 
$$
ae+bf\in  V' \ \text{iff} \ a=b=0,
$$
thus 
$$
V=V'\oplus {\rm Span}\{e,f\}.
$$
Since $\omega$ is non-degenerate, we know that for every $v\in V'$, there exists $u\in V$ such that $\omega(u, v)\neq 0$. Thus
$$
\omega(u', v)=\omega(u,v)\neq 0,
$$
which implies that $\omega|_{V'}$ is a symplectic form on $V'$. Thus the theorem follows by induction on $N$.
\end{proof}

One may use $\omega$ to define a bilinear form, say $\omega^{-1}$, on $V^*$ such that
$$
\omega^{-1}(f_j^*, e^*_k)=-\omega^{-1}(e^*_k, f_j^*)=\delta_{jk}, \ \omega^{-1}(f_j^*, f^*_k)=\omega^{-1}(e^*_j, e^*_k)=0.
$$
Let $T_\omega: V\to V^*$ be the linear isomorphism defined by
$$
T_\omega(u) (v)=\omega(v,u), \ \forall \ u,v\in V.
$$ 
Then we have
$$
T^{-1}_\omega=T_{\omega^{-1}},
$$
thus the definition of $\omega^{-1}$ does not depend on the choice of bases in the above proposition. We shall also use $\omega^{-1}$ to denote the following bilinear form on $\wedge^p V^*$:
\begin{equation}
\omega^{-1}(\mu, \nu):=\det(\omega^{-1}(\alpha_i, \beta_j)), \ \mu=\alpha_1\wedge \cdots \wedge \alpha_p, \ \nu=\beta_1\wedge \cdots \wedge \beta_p.
\end{equation} 

\begin{definition}[By Guillemin \cite{Guillemin-notes}] The symplectic star operator
$ *_s:  \wedge ^p V^* \to  \wedge ^{2n-p} V^*$ is defined by 
\begin{equation}
\mu\wedge *_s \nu=\omega^{-1}(\mu, \nu) \frac{\omega^n}{n!}.
\end{equation}
\end{definition}

The following theorem is the key to decode the structure of $*_s$.

\begin{theorem}[Hard Lefschetz theorem]\label{th:hlt-1} For each $0\leq k\leq n$, 
$$
u \mapsto \omega^{n-k} \wedge u, \ u\in \wedge^k V^*, 
$$
defines an isomorphism between $\wedge^k V^*$ and $\wedge^{2n-k} V^*$. 
\end{theorem}

\begin{proof} Notice that the theorem is true if $n=1$ or $k=0, n$. Now assume that it is true for $n\leq l$, $l\geq 1$. We need to prove that it is true for $n=l+1$, $1\leq k\leq l$. Put
$$
\omega'=\sum_{j=1}^l e_j^* \wedge f_j^*.
$$
Then we have
$$
\omega^{l+1-k}=(\omega')^{l+1-k} +(l+1-k)(\omega')^{l-k}\wedge  e_{l+1}^* \wedge f_{l+1}^*.
$$
Let us write $u\in \wedge^k(V^*)$ as
$$
u=u^0+u^1\wedge e_{l+1}^* +u^2\wedge f_{l+1}^*+u^3\wedge  e_{l+1}^* \wedge f_{l+1}^*,
$$
where each $u^j$ contains no $e_{l+1}^*$ or $f_{l+1}^*$ term.
Then $\omega^{l+1-k} \wedge u=0$ is equivalent to
$$
(\omega')^{l+1-k}  \wedge u^0=(\omega')^{l+1-k}  \wedge u^1=(\omega')^{l+1-k}  \wedge u^2=(\omega')^{l+1-k}  \wedge u^3+(l+1-k) (\omega')^{l-k}\wedge u^0=0,
$$
which implies $u^1=u^2=0$ by our theorem for $n=l$. Moreover, $u^3=0$ since
$$
(\omega')^{l+2-k}  \wedge u^3=(\omega')^{l+2-k}  \wedge u^3+(l+1-k) (\omega')^{l-k+1}\wedge u^0=0.
$$
Thus $ (\omega')^{l-k}\wedge u^0=0$, which implies $u^0=0$. Now we know that $u\mapsto \omega^{l+1-k}\wedge u$ is an injection, thus an isomorphism since  $\dim \wedge^k V^*=\dim \wedge^{2n-k} V^*$. 
\end{proof}

The key notion in these notes is the following:

\begin{definition}\label{de:primitive} We call $u\in \wedge^k V^*$ a primitive form if $k\leq n$ and $\omega^{n-k+1}\wedge u = 0$.
\end{definition}

The following Lefschetz decomposition theorem follows directly from Theorem \ref{th:hlt-1} (see the proof of Theorem \ref{th:Lef} in the next section).

\begin{theorem}[Lefschetz decomposition formula]\label{th: Lef-d-1} Every $u\in \wedge^k V^*$ has a unique decomposition as follows:
\begin{equation}
u=\sum \omega_r \wedge u^r, \ \omega_r:=\frac{\omega^r}{r!},
\end{equation}
where each $u^r$ is a primitive $(k-2r)$-form.
\end{theorem}

By the above theorem, it is enough to study the symplectic star operator on $\omega_r \wedge u$, where $u$ is primitive. 

\begin{theorem}\label{th: star-pri} If $u$ is a primitive $k$-form then $*_s(\omega_r \wedge u)=(-1)^{k+\cdots+1} \omega_{n-k-r} \wedge u$.
\end{theorem}

The above theorem implies $*_s^2=1$. We shall use a symplectic analogy of the Berndtsson lemma (see Lemma 3.6.10 in \cite{Bo-notes}) to prove it. 

\begin{definition} $u\in \wedge^k V^*$ is said to be an elementary form if there exists a base, say $\{e^*_1, f^*_1; \cdots; e^*_n, f^*_n\}$, of $V^*$ such that 
$$
\omega=\sum_{j=1}^n e^*_j\wedge f^*_j, \ u=e_1^*\wedge \cdots\wedge e_k^*.
$$
\end{definition}

\begin{lemma}[Berndtsson lemma] The space of primitive forms is equal to the linear space spanned by elementary forms.
\end{lemma}

\begin{proof} Since 
$$
\omega_{n-k+1}=\sum_{j_1< \cdots <j_{n-k+1}} e_{j_1}^*\wedge f_{j_1}^* \wedge \cdots \wedge e_{j_{n-k+1}}^*\wedge f_{j_{n-k+1}}^*,
$$
we know that $\omega_{n-k+1}\wedge u=0$ if $u$ is an elementary $k$-form. Thus every elementary form is primitive. 
Let us prove the other side by induction on $n$. Notice that the lemma is true if $n=1$. Assume that it is true for $n\leq l$, $l\geq 1$. We whall prove that it is also true for $n=l+1$. With the notation in the proof of Theorem \ref{th:hlt-1},  $\omega^{l-k+2} \wedge u=0$ is equivalent to
$$
(\omega')^{l+2-k}  \wedge u^0=(\omega')^{l+2-k}  \wedge u^1=(\omega')^{l+2-k}  \wedge u^2=(\omega')^{l+2-k}  \wedge u^3+(l+2-k) (\omega')^{l-k+1}\wedge u^0=0,
$$
which is equivalent to the $\omega'$-primitivity of $u^1, u^2, u^3$ and $(l+2-k)u^0+\omega'\wedge u^3$. Now it suffices to show that
$$
u':=u^3\wedge((l+2-k)e_{l+1}^*\wedge f_{l+1}^*-\omega')
$$ 
is a linear combination of elementary forms. Since $u^3$ is $\omega'$-primitive,  by the induction hypothesis, we can assume that
$$
u^3=e_1^*\wedge \cdots \wedge e_{k-2}^*.
$$
Thus
$$
u'=\sum_{j=k-1}^l e_1^*\wedge \cdots \wedge e_{k-2}^* \wedge (e_{l+1}^*\wedge f_{l+1}^*-e_{j}^*\wedge f_{j}^*).
$$
Now it suffices to show that if $n=2$ then $e_1^*\wedge f_1^*-e_2^*\wedge f_2^*$ is a linear combination of elementary forms. Notice that
$$
e_1^*\wedge f_1^*-e_2^*\wedge f_2^*=(e_1^*+e_2^*)\wedge(f_1^*-f_2^*)+e_1^*\wedge f_2^*+f_1^*\wedge e_2^*.
$$
It is clear that $e_1^*\wedge f_2^*$ and $f_1^*\wedge e_2^*$ are elementary. $(e_1^*+e_2^*)\wedge(f_1^*-f_2^*)$ is also elementary since we can write
$$
\omega=(e_1^*+e_2^*)\wedge f_1^*+e_2^*\wedge(f_2^*-f_1^*).
$$
The proof is complete.
\end{proof}

We shall also use the following Lemma from \cite{Guillemin-notes}.

\begin{lemma}[Guillemin Lemma] Assume that $(V,\omega)=(V_1, \omega^{(1)})\oplus(V_2, \omega^{(2)}) $. Then
$$
*_s(u\wedge v)=(-1)^{k_1k_2} *_s^1 u \wedge *_s^2 v,  \ u\in \wedge^{k_1}V_1^*, \ v\in \wedge^{k_2}V_2^*,
$$
where $*_s^1$ and $*_s^2$ are symplectic star operators on $V_1$ and $V_2$ respectively. 
\end{lemma}

\begin{proof} For every $a\in \wedge^{k_1}V_1^*, \ b\in \wedge^{k_2}V_2^*$, we have
$$
a\wedge  b \wedge (-1)^{k_1k_2} *_s^1 u \wedge *_s^2 v= \omega^{-1}(a\wedge u, b\wedge v)\omega_n,
$$
which gives the lemma.
\end{proof}

Now we are able to prove Theorem \ref{th: star-pri}.

\begin{proof}[Proof of Theorem \ref{th: star-pri}] By the Berndtsson lemma, we can assume that
$$
u=e_1^*\wedge \cdots \wedge e_k^*.
$$
Consider $V={\rm Span}\{e_j^*, f_j^*\}_{1\leq j\leq k} \oplus {\rm Span}\{e_{k+1}^*, f_{k+1}^*\} \oplus \cdots \oplus {\rm Span}\{e_{n}^*, f_{n}^*\} $ and write
$$
*_s=*_s^{\leq k}\oplus *_s^{k+1} \oplus\cdots \oplus *_s^n.
$$
Since
$$
*_s^{j}(1)=e_j^*\wedge f_j^*, \ *_s^{j}(e_j^*\wedge f_j^*)=1, \ \forall \ k+1\leq j\leq n,
$$
by the Guillemin lemma, we have
$$
*_s (e_{k+1}^*\wedge f_{k+1}^*\wedge \cdots \wedge e_{k+r}^*\wedge f_{k+r}^* \wedge u)=e_{k+r+1}^*\wedge f_{k+r+1}^*\wedge \cdots \wedge e_{n}^*\wedge f_{n}^* \wedge *_s^{\leq k} u,
$$
which implies
$$
*_s(\omega_r\wedge u)=\omega_{n-k-r} \wedge *_s^{\leq k} u.
$$
Since $*_s^{\leq k} = *_s^{1} \oplus\cdots \oplus *_s^k$ and
$$
*_s^{j} e_j^*=-e_j^*,  \ \forall \ 1\leq j\leq k,
$$
 the Guillemin lemma gives
$$
*_s^{\leq k} u=(-1)^{k-1} (-e_1^*) \wedge *_s^{\leq (k-1)} (e_2^*\wedge\cdots\wedge e_k^*)=\cdots=(-1)^{k+\cdots+1} u,
$$
the proof is complete.
\end{proof}

\begin{definition} We call $\{L, \Lambda, B\}$ the $sl_2$-triple on $\oplus_{0\leq k\leq 2n}\wedge^k V^*$,  where
$$
Lu:=\omega\wedge u, \ \Lambda :=*_s^{-1} L*_s, \ B:=[L,\Lambda].
$$
\end{definition}

We have
$$
\omega^{-1}(Lu,v)=\omega^{-1}(u, \Lambda v).
$$
Hence $\Lambda$ is the adjoint of $L$. Put
$$
L_r:=L^r/r!, \ L_0:=1, \ L_{-1}:=0.
$$
We have:

\begin{proposition}\label{pr:lamba-b} If $u$ is a primitive $k$-form then
$$
\Lambda(L_r u)=(n-k-r+1) L_{r-1} u, \ B(L_r u)=(k+2r-n)L_r u,
$$
for every $0\leq r \leq n-k+1$.
\end{proposition} 

\begin{proof} Put $c=(-1)^{k+\cdots+1}$, then
$$
L*_s(L_r u)=c L(L_{n-k-r} u)=(n-k-r+1)c L_{n-k-r+1} u=(n-k-r+1) *_s(L_{r-1} u),
$$
which gives the first identity. The second follows directly from the first.
\end{proof}

Now let us consider another structure on a linear space, which can be used to define an inner product structure on $(V, \omega)$. 

\begin{definition} We call a linear map $J: V\to V$ an almost complex structure on $V$ if $J(Ju)=-u$ for every $u\in V$.
\end{definition}

\begin{definition} An almost complex structure $J$ on $(V, \omega)$ is said to be compatible with $\omega$ if
$$
\omega(u, Jv)=\omega(v, Ju) , \ \forall \ u, v\in V, 
$$
and $\omega(u, Ju)>0$ if $u$ is not zero.
\end{definition}

If $J$ is an almost complex structure on $V$ then 
$$
J(v)(u):=v(Ju), \ \forall \ u\in V, \ v\in V^*,
$$
defines an almost complex structure on $V^*$. 

\begin{definition} We call 
$$
J(v_1\wedge \cdots\wedge v_k):=J(v_1)\wedge \cdots \wedge J(v_k), 
$$
the Weil operator on  $\oplus_{0\leq k\leq 2n}\wedge^k V^*$.
\end{definition}

Since the eigenvalues of $J$ are $\pm i$, its eigenvectors lie in $\C\otimes V^*$. Put
$$
E_i:=\{u\in \C\otimes V^*: J(u)=iu \}, \ \ E_{-i}:=\{u\in \C\otimes V^*: J(u)=-iu \},
$$
we know that
$$
E_i=\{u-iJu: u\in V^*\}, \ \ E_{-i}=\{u+iJu: u\in V^*\}.
$$
and $\C\otimes V^*=E_{i}\oplus E_{-i}$. Put
$$
\wedge^{p,q} V^*:=(\wedge^p E_i) \wedge (\wedge^q E_{-i}).
$$
Then we have
$$
\C\otimes(\wedge^k V^*)=\wedge^k(\C\otimes V^*)=\oplus_{p+q=k} \wedge^{p,q} V^*,
$$
and
$$
Ju=i^{p-q} u, \ \forall \ u\in \wedge^{p,q} V^*.
$$
We call $\wedge^{p,q} V^*$ the space of $(p,q)$-forms. 

\begin{proposition} An almost complex structure $J$ on $(V, \omega)$ is compatible with $\omega$ iff
$$
(\alpha,\beta):=\omega^{-1}(\alpha, J \bar \beta), 
$$
defines a Hermitian inner product structure on $\wedge^{p,q} V^*$, $0\leq p,q \leq n$.
\end{proposition}

\begin{proof} Assume that $J$ is compatible with $\omega$. Then 
$$
T_{\omega}(Ju)(v)=\omega(v, Ju)=-\omega(Jv,u)=-J(T_{\omega} u)(v), \ u,v\in V,
$$
Thus $T_\omega \circ J=-J\circ T_{\omega}$. Now put 
$$
a=T_{\omega}(u), \ b=J(T_{\omega}(v))=-T_{\omega}(Jv).
$$
Then $\omega(v, Ju)=T_{\omega}(Ju)(v)=-(Ja)(v)=-a(Jv)=a(T_{\omega^{-1}}(b))$, thus
$$
\omega(v, Ju)=T_{\omega^{-1}}(b)(a)=\omega^{-1}(a,b)=\omega^{-1}(T_{\omega} u, J(T_{\omega} v)),
$$
which gives the proposition.
\end{proof}

\begin{definition} The Hodge star operator $*: \wedge^{p,q} V^* \to \wedge^{n-q,n-p} V^*$ is defined by
$$
u\wedge *\bar v=(u,v)\omega_n.
$$
\end{definition}

The above proposition gives
$$
*=*_s\circ J=J\circ *_s.
$$

\subsection{Application in complex geometry}

Let $(X,\omega)$ be an $n$-dimensional complex manifold with a \emph{Hermitian form} (smooth positive $(1,1)$-form) $\omega$. Let $(E,h_E)$ be a holomorphic vector bundle over $X$ with a smooth Hermitian metric $h_E$. Let us denote by $V^{k}$ the space of $E$-valued $k$-forms with compact support on $X$. The following theorem is a direct consequence of Theorem \ref{th:hlt-1}.

\begin{theorem}[Hard Lefschetz theorem] For each $0\leq k\leq n$, 
\begin{equation}\label{eq:hl}
u \mapsto \omega^{n-k} \wedge u,  \  \ u\in V^{k},
\end{equation}
defines an isomorphism between $V^{k}$ and $V^{2n-k}$.
\end{theorem}

\begin{definition} We call an $E$-valued $k$-form, say $u$, on $X$ a primitive form if $k\leq n$ and $\omega^{n-k+1}\wedge u\equiv 0$.
\end{definition}

Now we have the following analogy of Theorem \ref{th: Lef-d-1}:

\begin{theorem}[Lefschetz decomposition formula] Every $E$-valued $k$-form $u$ on $X$ has a unique decomposition as follows:
\begin{equation}
u=\sum \omega_r \wedge u^r, \ \omega_r:=\frac{\omega^r}{r!},
\end{equation}
where each $u^r$ is an $E$-valued primitive $(k-2r)$-form.
\end{theorem}

Let $\{e_\alpha\}$ be a local holomorphic frame of $E$, then
$$
||u||^2:=\int_X \sum h_E(e_\alpha, e_\beta) u^\alpha \wedge * \bar u^\beta, \ u:=\sum u^\alpha\otimes e_\alpha.
$$ 
defines a Hermitian inner product structure on $V^k$, we call it  $(\omega,J, h_E)$-metric on $V^k$.

\begin{definition} The Hodge star operator on $V^k$ is defined by $$*u=\sum (*u^\alpha)\otimes e_\alpha, \ u:=\sum u^\alpha\otimes e_\alpha. $$
\end{definition}

Put
$$
\left\{\sum u^\alpha\otimes e_\alpha, \sum u^\beta\otimes e_\beta\right\}_{h_E}:= \sum h_E(e_\alpha, e_\beta)  u^\alpha \wedge \bar u^\beta. 
$$
Then we have
$$
||u||^2=\int_X \{u, *u\}_{h_E}.
$$
The Hodge-Riemann bilinear relation is a direct consequence of Theorem \ref{th: star-pri} and $*=J\circ *_s$.

\begin{theorem}[Hodge-Riemann bilinear relation] If $u$ is an $E$-valued primitive $(p,q)$-form then its $(\omega, J, h_E)$-norm satisfies
\begin{equation}\label{eq:HRR}
||u||^2=\int_X \{u, \omega_{n-k} \wedge I u\}_{h_E},\ Iu:=(-1)^{k+\cdots+1}i^{p-q} u, \ k:=p+q.
\end{equation}
\end{theorem}

\section{Lefschetz bundle}

\begin{definition}
Let $V=\oplus_{k=0}^{2n} V^k$ be a direct sum of complex vector bundles over a smooth manifold $M$. Let $L$ be a smooth section of ${\rm End}(V)$. We call $(V, L)$ a Lefschetz bundle if
$$
L(V^l)\subset V^{l+2}, \ \forall \ 0\leq l\leq 2(n-1), \ L(V^{2n-1})=L(V^{2n})=0,
$$ 
and each $L^k: V^{n-k} \to V^{n+k}$, $0 \leq k\leq n$, is an isomorphism. 
\end{definition}

\begin{definition} Let $(V, L)$ be a Lefschetz bundle.  $u\in V^k$ is said to be primitive if  $k\leq n$ and $L^{n-k+1} u=0$. 
\end{definition}

\begin{theorem}\label{th:Lef}  Let $(V, L)$ be a Lefschetz bundle. Then every $u\in V^k$ has a unique decomposition as follows:
\begin{equation}
u=\sum L_r  u^r, \ L_r:=\frac{L^r}{r!}.
\end{equation}
where each $u^r$ is a primitive form in $V^{k-2r}$.
\end{theorem}

\begin{proof} We can assume that $k\leq n$ since we have the isomorphism $L^k: V^{n-k} \to V^{n+k}$. Notice that the theorem is trivial if $k=0,1$. Assume that $2\leq k \leq n$. The  isomorphism 
$$
L^{n-k+2}: V^{k-2}\to V^{2n-k+2},
$$ 
gives $\hat u\in V^{k-2}$ such that $L^{n-k+2}\hat u=L^{n-k+1} u$. Put $u^0=u-L\hat u$, we know that $u^0$ is primitive and $u=u^0+L\hat u$. Consider $\hat u$ instead of $u$, we have
$\hat u= u^1 +L \tilde u$, where $u^1$ is primitive. By induction, we know that $u$ can be written as
$$
u=\sum L_r u^r,
$$
where each $u^r\in V^{k-2r}$ is primitive. For the uniqueness part, assume that
$$
0=\sum_{r=0}^j L_r u^r,
$$
where each $u^r\in V^{k-2r}$ is primitive. Then we have
$$
0=L_{n-k+j} \left(\sum_{r=0}^j L_r u^r\right) = L_{n-k+j}L_j u^j,
$$
which gives $u^j =0$. By induction on $j$ we know that all $u^r=0$.
\end{proof}

\begin{definition} We call the following $\C$-linear map $*_s: V\to V$ defined by
$$
*_s(L_r u):=(-1)^{k+\cdots+1} L_{n-r-k} u,
$$
where $u\in V^k$ is primitive, the Lefschetz star operator on $V$. 
\end{definition}

Notice that $*_s^2=1$. We know from the last section that the Lefschetz star operator is a generalization of the symplectic star operator.

\begin{definition} Put $\Lambda=*_s^{-1} L*_s$, $B:=[L, \Lambda]$. We call $(L,\Lambda, B)$ the $sl_2$-triple on $(V,L)$ (Proposition \ref{pr:lamba-b} is also true for general Lefschetz bundle).
\end{definition}

\section{Variation of Lefschetz star operator}

\subsection{Main theorem}

Our main theorem is a generalization of the main result in \cite{Wang-k}.

\begin{theorem}\label{th:main}  Let $D$ be a degree preserving connection on a Lefschetz bundle $(V,L)$. Put $ \theta:=[D, L]$.  If $[L,\theta]=0$ then $*_s^{-1}D*_s=D+[\Lambda, \theta]$. 
\end{theorem}

\begin{proof} By the Lefschetz decompostion theorem and $*_s^2=1$, it suffices to prove 
$$
*_s D*_s(L_r u)=D(L_r u)+[\Lambda, \theta](L_r u),
$$
where $u$ is a primitive $k$-form. Since $[L,\theta]=0$, we have
$$
D*_s(L_r u)=cD(L_{n-r-k} u)=c (L_{n-r-k-1} \theta u+L_{n-r-k} Du), \  \ c:=(-1)^{k+\cdots+1}.
$$

\emph{Step 1}: Since $u$ is primitive, we have
$$
L_{n-k+1}\theta u= \theta L_{n-k+1} u =0,
$$
which implies that the primitive decomposition of $\theta u$ contains at most three terms. Thus we can write
\begin{equation}
\theta u=a+L b+L^2 c,
\end{equation}
where $a, b, c$ are primitive, which gives
$$
*_s D *_s(L_r u)= -L_{r-1}a
+M L_r b-M(M+1) L_{r+1} c+  c*_s L_{n-r-k} Du, \ M:= n-r-k.
$$

\emph{Step 2}: Since
\begin{equation}
\theta L_r u=L_r(a+L b+L^2 c), 
\end{equation}
Proposition \ref{pr:lamba-b} gives
$$
\Lambda\theta L_r u = (M-1)L_{r-1} a +(r+1)M L_{r} b + (r+1)(r+2)(M+1)L_{r+1} c,
$$
and
$$
\theta \Lambda L_r u = (M+1) \theta  L_{r-1} u =  (M+1)L_{r-1} (a+L b+L^2 c).
$$
Thus
\begin{equation}
[\Lambda, \theta]  L_r u=-2 L_{r-1} a+(M-r) L_r b+(2r+2)(M+1)L_{r+1} c.
\end{equation}

\emph{Step 3}:  Put
$$
A:= *_s D *_s(L_r u)-[\Lambda, \theta]  L_r u, \ B:= D(L_r u).
$$
We have
\begin{eqnarray*}
B  & =  &  L_{r-1}\theta u+L_r D u \\
& =  &   L_{r-1}a+ r L_r b + r(r+1) L_{r+1} c + L_r D u.
\end{eqnarray*}
Since the first two steps gives
$$
A=L_{r-1} a + rL_r b -(M+1)(M+2r+2)L_{r+1} c+  c*_s L_{n-r-k} Du, 
$$
we have
$$
A-B=c*_s L_{n-r-k} Du-  L_r D u- [(M+1)(M+2r+2)+r(r+1)] L_{r+1} c .
$$

\emph{Step 4}:  Primitivity of $u$ implies
$$
0=D(L_{n-k+1} u)=\theta L_{n-k} u + L_{n-k+1} Du.
$$
Notice that 
$$
\theta  L_{n-k} u=L_{n-k} (a+Lb+L^2 c) =L_{n-k}L^2 c.
$$
Thus 
$$
L_{n-k+1} (Du +(n-k+1)L c) = 0,
$$
which implies the primitivity of 
\begin{equation}
v:=Du +(n-k+1)L c.
\end{equation}
Now we have
\begin{eqnarray*}
c*_s L_{n-r-k} Du & = & c*_s L_{n-r-k}  (v-(n-k+1)L c) \\
&  = & L_r v+(n-k+1)(M+1)L_{r+1} c .
\end{eqnarray*}
and
\begin{equation}
L_{r} D u=L_{r} (v-(n-k+1) L c)
= L_r v- (n-k+1)(r+1) L_{r+1} c.
\end{equation}
Thus $c*_s L_{n-r-k} Du  - L_{r} D u$ can be written as
\begin{equation}
[(n-k+1) (M+1)+ (n-k+1)(r+1)] L_{r+1} c.
\end{equation}

\emph{Step 5}: Since our formula is equivalent to $A=B$, by step 3 and 4, it is enough to prove
$$
(n-k+1) (M+1)+ (n-k+1)(r+1)=(M+1)(M+2r+2)+r(r+1),
$$
which is true (recall that $M=n-r-k$).
\end{proof}

\textbf{Remark}: If we write
$$
D=\sum dt^j \otimes D_j,
$$
where $\{t^j\}_{1\leq j\leq m}$ are smooth local coordinates. Then our main theorem is equivalent to
\begin{equation}\label{eq:wangxu}
*_s^{-1} D_j*_s u=D_ju+[\Lambda, \theta_j] u, \ \ \theta_j:=[D_j, L],
\end{equation}
for every $1\leq j\leq m$.

\subsection{Corollary}

\begin{corollary} With the same assumption in Theorem \ref{th:main}, we have 
\begin{equation}\label{eq:main-1}
*_s^{-1} \theta *_s =  [\Lambda, D].
\end{equation}
\end{corollary}

\begin{proof} The lemma below gives
\begin{equation}\label{eq:main-2}
*_s^{-1} \theta*_s =  -\frac12[\Lambda, [\Lambda, \theta]].
\end{equation}
Since $\Lambda=*_s^{-1} L*_s$ and $\theta=[D, L]=DL-LD$, we have
$$
*_s^{-1} \theta*_s =[*_s^{-1} D*_s, \Lambda]. 
$$
Our main theorem and \eqref{eq:main-2} give
$$
*_s^{-1} \theta*_s =[D+[\Lambda, \theta], \Lambda]=[D,\Lambda]+2 *_s^{-1} \theta*_s,
$$
which gives \eqref{eq:main-1}. 
\end{proof}

\begin{lemma} $*_s^{-1} \theta*_s =  -\frac12[\Lambda, [\Lambda, \theta]]$.
\end{lemma}

\begin{proof} For a primitive $k$-form $u$,  we have
$\theta *_s(L_r u):= c L_{n-r-k} \theta  u$. 
By the proof of our main theorem, we can write $
\theta u=a+L b+L^2 c$, where $a, b, c$ are primitive. Thus
$$
*_s\theta*_s(L_r u)=-L_{r-2}a+(M+1)L_{r-1}b-(M+1)(M+2)L_r c, \ \  M:=n-r-k.
$$
On the other hand, notice that
$$
-\frac12 [\Lambda, [\Lambda, \theta]]=\Lambda \theta\Lambda-\frac12 (\Lambda^2\theta+\theta\Lambda^2).
$$
Now 
$$
\Lambda^2 \theta L_r u =\Lambda^2 L_r(a+L b+L^2 c).
$$
thus Proposition \ref{pr:lamba-b} gives
\begin{eqnarray*}
\Lambda^2 \theta L_r u  & = & (M-1)M L_{r-2}a+(r+1) M(M+1)L_{r-1}b \\
   &  & +(r+1)(r+2)(M+1)(M+2) L_{r}c.
\end{eqnarray*}
By a similar argument, we also get
$$
\theta\Lambda^2  L_r u=(M+1)(M+2)(L_{r-2}a+ (r-1)L_{r-1}b +r(r-1)L_{r}c),
$$
and
$$
\Lambda \theta \Lambda  L_r u=(M+1)(M L_{r-2}a+ r(M+1)L_{r-1}b +r(r+1)(M+2)L_{r}c).
$$
Thus $-\frac12 [\Lambda, [\Lambda,\theta]](L_ru)$ equals
$$
-L_{r-2}a +(M+1)L_{r-1}b-(M+1)(M+2)L_r c=*_s\theta*_s(L_r u).
$$
The proof is complete.
\end{proof}

The following proposition can be seen as a generalization of  formula 1 in \cite{OT87}.

\begin{proposition}\label{pr: formula-1} Let $(V, L)$ be a Lefschetz bundle. Let $\sigma$ be a smooth degree one section of ${\rm End}(V)$. If $[L, \sigma]=0$ then $*_s^{-1}\sigma *_s=(-1)^k[\Lambda, \sigma]$ on $V^k$.
\end{proposition}

\begin{proof}  For a primitive $k$-form $u$,  we have
$\sigma *_s(L_r u):= c L_{n-r-k} \sigma  u$. 
Since $\sigma$ is degree one, we can write $
\sigma u=a+L b$, where $a, b$ are primitive. Thus
$$
*_s\sigma*_s(L_r u)=(-1)^{k+1} (L_{r-1}a-(M+1)L_{r}b), \ \  M:=n-r-k.
$$
On the other hand, Proposition \ref{pr:lamba-b} gives
$$
\Lambda \sigma L_r u =M L_{r-1} a+(r+1)(M+1)L_r b.
$$
and
$$
\sigma \Lambda L_r u=(M+1)(L_{r-1}a+ rL_{r}b).
$$
Thus 
$$
(-1)^k[\Lambda, \sigma](L_r u)= *_s\sigma*_s(L_r u).
$$
Since $*_s^2=1$, the proposition follows.
\end{proof}

\section{$T$-Hodge theory}

\subsection{Timorin's theorem} Timorin's theorem \cite{Timorin98} is a mixed linear version of the Hodge-Riemann bilinear relation. Let $(V, \omega, J)$ be a $2n$-dimensional real vector space with compatible pair $(\omega, J)$. Let $\alpha_0, \alpha_1, \cdots, \alpha_n$ be $J$-compatible symplectic forms on $V$. Put
$$
 T_k:=\alpha_{k}\wedge \cdots\wedge \alpha_n, \ 0\leq k\leq n, \ \ T_{n+1}:=1.
$$
and
$$
V^k:=\C\otimes \wedge^k V^*.
$$
Timorin introduced the following definition in \cite{Timorin98}.

\begin{definition} We call $u\in V^k$ a $T_k$-primitive form if $k\leq n $ and $u\wedge T_k=0$.
\end{definition}

\begin{theorem}[Timorin's mixed Hodge-Riemann bilinear relation, MHR-n] Let $u$ be a non-zero $T_k$-primitive form. Then 
$Q(u):=(-1)^{k+\cdots+1} T_{k+1} \wedge u\wedge \overline{Ju} >0$.
\end{theorem}

\begin{proof} We claim that MHR-n follows MHL-n and usual Hodge-Riemann bilinear relation. Notice that MHL-n below implies that the space, say $P_k$, of $T_k$-primitive forms has constant dimension $\dim V^k-\dim V^{k-2}$ and $Q$ is non-degenerate on $P_k$, consider
$$
\alpha_j^t:=(1-t)\alpha_j + t\omega, \ 0\leq t\leq 1,
$$
then the positivity of $Q$ at $t=0$ follows from the positivity of $Q$ at $t=1$ (the usual Hodge-Riemann bilinear relation). 
\end{proof}

\begin{theorem}[Timorin's mixed hard-Lefschetz theorem, MHL-n] For every $0\leq k\leq n$,
$$
u\mapsto u \wedge T_{k+1},
$$
defines an isomorphism from $V^k$ to $V^{2n-k}$.
\end{theorem}

\begin{proof} Since MHR-1 is true, it suffices to show MHR-(n-1) implies MHL-n. Assume that $u\in V^k$, $k\leq n-1$. If $u \wedge T_{k+1}=0$ then
$$
u|_{H}\wedge T_{k+1}|_{H} =0,
$$
for every hyperplane $H$. Thus if $u \wedge T_{k+1}=0$ then
$u|_{H}$ is $T_{k+1}|_{H}$-primitive for every $H$. Let us write
$$
\alpha_{k+1}=\sum_{j=1}^n i \sigma_j \wedge \bar \sigma_j, \  H_j:=\ker \sigma_j.
$$ 
Then MHR-(n-1) gives
$$
Q_j(u):=(-1)^{k+\cdots+1} T_{k+2}|_{H_j} \wedge u|_{H_j}\wedge \overline{Ju|_{H_j}} \geq 0.
$$
If $u \wedge T_{k+1}=0$ then 
$$
0=(-1)^{k+\cdots+1} T_{k+1} \wedge u\wedge \overline{Ju}=\sum Q_j(u) \wedge (i\sigma_j\wedge \bar \sigma_j),
$$
which implies each $Q_j(u) = 0$. Thus $u|_{H_j}=0$ for every $1\leq j \leq n$, which gives $u\wedge \alpha_{k+1}=0$, thus $u=0$ since $\deg u \leq n-1$.
\end{proof}

\subsection{Hodge star operator on $V_T$} Let $(E, h_E)$ be a holomorphic vector bundle over an $n$-dimensional complex manifold $(X, \omega)$. Denote by $V^{p,q}$ the space of smooth $E$-valued $(p,q)$-forms with compact support on $X$. Put
$$
V:=\oplus_{0\leq p,q \leq n} V^{p,q}, \ V^k:=\oplus_{ p+q = k} V^{p,q}.
$$
Fix $0\leq m\leq n$ and smooth positive $(1,1)$-forms $\alpha_{m+1}, \cdots, \alpha_n$ on $X$. Consider
$$
f_T: u\mapsto T\wedge u, \ u\in V.
$$
where 
$$
T:=\alpha_{m+1} \wedge \cdots \wedge \alpha_n,\ T:=1, \ \text{if}\ m=n.
$$
We call the Hodge theory on 
$$
{\rm Im} f_T:=\{T\wedge u: u\in V\},
$$
the $T$-Hodge theory. Put
$$
V_T^{p,q}:= f_T(V^{p,q}), \ V_T:=\oplus_{0\leq p,q \leq n} V_T^{p,q}, \ V_T^k:=\oplus_{ p+q = k} V_T^{p,q}.
$$
We have
$$
V_T=\oplus_{k=0}^{2m} V_T^{k}=\oplus_{0\leq p,q \leq m} V_T^{p,q},
$$
and
$$
L: u\mapsto \omega\wedge u, \ u\in V_T,
$$
maps $V_T^{p,q}$ to $V_T^{p+1,q+1}$. Timorin's mixed hard-Lefschetz theorem gives:

\begin{theorem}\label{th: T-HL} For every $0\leq k\leq m$,
$$
L^{m-k}: u\mapsto u\wedge\omega^{m-k},
$$
defines an isomorphism from $V^k_T$ to $V_T^{2m-k}$.
\end{theorem}

\begin{proof} By Timorin's theorem, we know that 
$$
u\mapsto T\wedge u, \ u\in V^k
$$
is injective. Thus $f_T$ defines an isomorphism from $V^k$ to  $V^k_T$. Again by Timorin's theorem, we have the following isomorphism 
$$
A: u\mapsto u\wedge T\wedge \omega^{m-k},
$$
from $V^k$ to $V^{2n-k}$. Thus $L^{m-k}=A\circ f^{-1}_T$ is an isomorphism.
\end{proof}

\begin{definition} We call $u\in V^k_T$ a primitive $k$-form if $k\leq m$ and $L^{m-k+1} u=0$.
\end{definition}

The proof of Theorem \ref{th:Lef} implies:

\begin{theorem}\label{th: T-lef} Every $u\in V_T^k$ has a unique decomposition as follows:
\begin{equation}
u=\sum L_r  u^r, \ L_r:=\frac{L^r}{r!}.
\end{equation}
where each $u^r$ is a primitive form in $V_T^{k-2r}$.
\end{theorem}

\begin{definition} We call the following $\C$-linear map $*_s: V_T\to V_T$ defined by
$$
*_s(L_r u):=(-1)^{k+\cdots+1} L_{m-r-k} u,
$$
where $u\in V_T^k$ is primitive, the Lefschetz star operator on $V_T$. 
\end{definition}

In case $m=n$,  the Lefschetz star operator above is just the symplectic star operator.

\begin{definition} Put $\Lambda=*_s^{-1} L*_s$, $B:=[L, \Lambda]$. We call $(L,\Lambda, B)$ the $sl_2$-triple on $V_T$.
\end{definition}

Since $J$ commutes with $f_T$, the Weil-operator is also well defined on $V_T$, we shall also denote it by $J$. 

\begin{definition} We call $*:=*_s\circ J$ the Hodge star operator on $V_T$. 
\end{definition}

Timorin's mixed Hodge-Riemann bilinear relation gives:

\begin{theorem}\label{th: T-metric} Put
$$
(u,v):=\int_X\{f_T^{-1}(u), *v\}_{h_E}, \ \ u, v\in V^k_T, \ 0\leq k\leq m,
$$
and
$$
(u,v):=\int_X\{u, f_T^{-1}(*v)\}_{h_E}, \ \ u, v\in V^k_T, \ m\leq k\leq 2m.
$$
Then $(u,v)$ is a Hermitian inner product on $V_T$.
\end{theorem}

\begin{definition} Let us define $(u, v)_T$ such that
$$
\{f_T^{-1}(u), *v\}_{h_E}=(u, v)_T ~ \omega_m \wedge T,  \ \ u, v\in V^k_T, \ 0\leq k\leq m,
$$
and
$$
\{u, f_T^{-1}(*v)\}_{h_E}=(u, v)_T ~ \omega_m \wedge T, \ \ u, v\in V^k_T, \ m\leq k\leq 2m.
$$
We call $(u,v)_T$ the pointwise Hermitian inner product of $u, v$ in $V_T$.
\end{definition}

\textbf{Remark}: If $m=n$ then $T=1$ and $(u, v)_T$ is just the pointwise $(\omega, J, h_E)$-inner product. Moreover, our Hermitian metric in Theorem \ref{th: T-metric} is compatible with the current norm on $V_T^{p,0}$ defined by Berndtsson-Sibony in \cite{BS}. 

\subsection{K\"ahler identity in $T$-Hodge theory}

Let 
$$
D:=\dbar+\partial^E,
$$
be the Chern connection on $(E, h_E)$. Assume that $T$ is $d$-closed, then $D u\in V_T$ if $u\in V_T$. Let $D^*$, $\dbar^*$, $(\partial^E)^*$ be the adjoint of
$$
D, \ \dbar, \  \partial^E: V_T\to V_T,
$$
respectively. We shall use Theorem \ref{th:main} and Proposition \ref{pr: formula-1} to prove the following $T$-geometry generalization of the Demailly-Griffiths-K\"ahler identity (see page 307 in \cite{Demailly12}).

\begin{theorem}\label{th: T-KID} If $T$ is $d$-closed then $[\dbar^*, L]=i(\partial^E+[\Lambda, [\partial^E, L]])$ on $V_T$.
\end{theorem}

\begin{proof} Since $*=*_s\circ J=J\circ *_s$ and $*_s^2=1$, we have
$$
\dbar^*=-*\partial^E *=(-1)^{k+1} i*_s\partial^E*_s=(-1)^{k+1} i \sum_{j=1}^n (*^{-1}_s \partial^E_j*_s) (*^{-1}_s \sigma_j*_s),
$$
on $V_T^k$, where $\sigma_j:=dz^j \wedge$. Thus
$$
[\dbar^*, L]=(-1)^{k+1} i \sum_{j=1}^n (*^{-1}_s \partial^E_j*_s) (*^{-1}_s \sigma_j*_s)L- L(*^{-1}_s \partial^E_j*_s) (*^{-1}_s \sigma_j*_s).
$$
Now Proposition \ref{pr: formula-1} implies $[*^{-1}_s \sigma_j*_s, L]=(-1)^{k+1}\sigma_j$, thus
\begin{eqnarray*}
[\dbar^*, L]  & = &  (-1)^{k+1} i \sum_{j=1}^n (*^{-1}_s \partial^E_j*_s) ((-1)^{k+1}\sigma_j +L*^{-1}_s \sigma_j*_s)- L(*^{-1}_s \partial^E_j*_s) (*^{-1}_s \sigma_j*_s) \\
& = &   i \sum_{j=1}^n (*^{-1}_s \partial^E_j*_s) \sigma_j +(-1)^{k+1} [*^{-1}_s \partial^E_j*_s, L] (*^{-1}_s \sigma_j*_s).
\end{eqnarray*}
\eqref{eq:main-1} gives $[*^{-1}_s \partial^E_j*_s, L]=-\theta_j$, where $\theta_j:=[\partial^E_j, L]$, thus our main theorem gives
$$
[\dbar^*, L]=i \partial^E +i\sum [\Lambda,\theta_j]\sigma_j + (-1)^k\theta_j (*^{-1}_s \sigma_j*_s).
$$
Now it suffices to show 
\begin{equation}\label{eq:DGK}
\sum [\Lambda,\theta_j]\sigma_j + (-1)^k\theta_j (*^{-1}_s \sigma_j*_s)=[\Lambda, [\partial^E, L]].
\end{equation}
Since $\sum \theta_j \sigma_j=[\partial^E, L]$, we have
\begin{equation}
\sum [\Lambda,\theta_j]\sigma_j=\Lambda [\partial^E, L]-\sum \theta_j\Lambda \sigma_j.
\end{equation}
Proposition \ref{pr: formula-1} gives
\begin{equation}
(-1)^k\theta_j (*^{-1}_s \sigma_j*_s)=\theta_j \Lambda \sigma_j-\theta_j \sigma_j \Lambda.
\end{equation}
Thus the left hand side of \eqref{eq:DGK} can be written as
\begin{equation}
\Lambda [\partial^E, L]-[\partial^E,L]\Lambda,
\end{equation}
which equals the right hand side of \eqref{eq:DGK}. 
\end{proof}

\begin{theorem}\label{th:T-kahler} Assume that both $T$ and $\omega$ are $d$-closed, then 
$$
D^*= [\Lambda,D^c], \ (D^c)^*=[D, \Lambda],
$$
on $V_T$, where $D^c:=i\dbar-i\partial^E$. 
\end{theorem}

The above K\"ahler identity gives the following Bochner-Kodaira-Nakano identity in $T$-geometry:

\begin{theorem}\label{th:T-BKN} Assume that both $T$ and $\omega$ are $d$-closed. Then 
$$
\Box_{\dbar}= \Box_{\partial^E} +[i\Theta(E, h_E), \Lambda], 
$$
on $V_T$, where $\Theta(E, h_E):=D^2, \ \Box_{\dbar}:=\dbar^*\dbar + \dbar\dbar^*, \ \Box_{\partial^E}:= \partial^E(\partial^E)^* + (\partial^E)^*\partial^E$; moreover
$$
\Box_{D^c}= \Box_{D}=\Box_{\dbar}+\Box_{\partial^E},
$$
on $V_T$, where $\Box_{D}:=DD^*+D^*D, \ \Box_{D^c}:=D^c(D^c)^*+(D^c)^*D^c$. 
\end{theorem}

\textbf{Remark}: If $1\leq m<n$ then $\Box_{D}$ is elliptic on $V_T^{k,0}$ and $V_T^{0,k}$ for $0\leq k\leq m-1$, but it is not elliptic on $V_T^{m,0}$ and $V_T^{0,m}$. In general, we don't have
$$
(\dbar f)^*(\dbar f) +(\dbar f) (\dbar f)^*= |\dbar f|^2_T,
$$
but still we have the following theorem:

\begin{theorem}\label{th:elliptic} $\Box_{\dbar}, \Box_{\partial^E}$ are elliptic on $V_T^k$ for every $k\neq m$.
\end{theorem}

\begin{proof} Recall that a differential operator of order $l$ is said to be elliptic if $\sigma_l (D)(x, \xi)$ is invertible for every $x\in M$ and every non-zero $\xi \in T_x M$, where
$$
\sigma_l (D)(x, \xi)u:=\lim_{t\to \infty} t^{-l} e^{-itf} D (e^{itf}u) (x),
$$ 
and $f$ is a smooth function near $x$ such that $df(x)=\xi$. We have
$$
\sigma_2(\Box_{\dbar})(x, \xi)= (\dbar f)^*(\dbar f) +(\dbar f) (\dbar f)^*, \ \sigma_2(\Box_{\partial^E})(x, \xi)= (\partial f)^*(\partial f) +(\partial f) (\partial f)^*.
$$
Theorem \ref{th:T-BKN} gives 
$$
\sigma_2(\Box_{\dbar})(x, \xi)=\sigma_2(\Box_{\partial^E})(x, \xi).
$$
Thus it suffices to prove that if $u\in V_T^k$, $k\neq m$, satisfies 
$$
(\dbar f)\wedge u= (\partial f) \wedge u= (\dbar f)^* u =(\partial f)^* u=0,
$$
then $u(x)=0$. It is clear that we can assume that $u\in V_T^{p,q}$, $p+q=k$. Consider $*u$ if $k>m$, one may assume further that $k<m$. Moreover, by a $\C$-linear change of local coordinate, one may assume that $\dbar f=d\bar z_1, \ \partial f=dz_1$. Let
$$
u=\sum_{r=0}^{j} \omega_r \wedge u^r,
$$
be the Lefschetz decomposition of $u$. Then
$$
*u=(-1)^{k+\cdots+1} i^{p-q} \sum_{r=0}^{j} (-1)^r \omega_{m-k+r} \wedge u^r.
$$
The lemma below gives 
$$
u^r\wedge dz_1=u^r\wedge d\bar z_1=0, \ \forall \ 0\leq r\leq j,
$$
hence
$$
f_T^{-1}(u^r) \wedge dz_1=f_T^{-1}(u^r) \wedge d\bar z_1=0,
$$
since the degree of each $u^r$ is no bigger than $m-1$. Thus each $f_T^{-1}(u^r)$ can be written as
$$
dz_1\wedge d\bar z_1 \wedge v^r.
$$
Timorin's Hodge-Riemann bilinear relation implies that if $f_T^{-1}(u^r)  \neq 0$ then $f_T^{-1}(u^r) \wedge \overline{f_T^{-1}(u^r)} \neq 0$. But obviously $(dz_1\wedge d\bar z_1)^2=0$ gives $f_T^{-1}(u^r) \wedge \overline{f_T^{-1}(u^r)} = 0$. Thus $f_T^{-1}(u^r)  = 0$ and $u=0$. The proof is complete.
\end{proof}

\begin{lemma} $u^r\wedge dz_1=u^r\wedge d\bar z_1=0, \ \forall \ 0\leq r\leq j$.
\end{lemma}

\begin{proof} $dz_1 \wedge u=dz_1\wedge *u=0$ gives
$$
A:=\sum_{r=0}^{j} \omega_r \wedge (dz_1\wedge u^r) =0, \ B:= \sum_{r=0}^{j} (-1)^r \omega_{m-k+r} \wedge (dz_1\wedge u^r)=0. 
$$
The lemma is true if $j=0$. Assume the lemma is true for $0\leq j\leq l$. We claim that it is true for $j=l+1$. In fact,  $\omega_{m-k} \wedge A + (-1)^{l+1}C_{m-k+l}^l B=0$ gives
$$
C:=(\omega_{m-k} \wedge \omega_{l+1}+ C_{m-k+l}^l \omega_{m-k+l+1})\wedge (dz_1\wedge u^{l+1})+
\sum_{r=0}^{l-1} C_r \omega_{m-k+r} \wedge (dz_1\wedge u^r)=0. 
$$
Since each $u^r$ is primitive, we know that 
$$
\omega_l\wedge C =\omega_r \wedge (\omega_{m-k} \wedge \omega_{l+1}+ C_{m-k+l}^l \omega_{m-k+l+1})\wedge (dz_1\wedge u^{l+1}),
$$
which gives $dz_1\wedge u^{l+1}=0$ by Timorin's hard Lefschetz theorem. Thus the lemma is also true for $j=l+1$ by the induction hypothesis. 
\end{proof}

\section{Applications}

\subsection{Alexandrov-Fenchel inequality} 

\begin{definition} A Hermitian manifold $(X, \hat \omega)$ is said to be complete if there exists a smooth function, say 
$$
\phi: X\to [0, \infty),
$$
such that $\phi^{-1}([0, c])$ is compact for every $c>0$ and
$$
|d\phi|_{\hat \omega} (x) \leq 1, \ \forall \ x\in X.
$$
\end{definition}

\begin{theorem}\label{th:T-AF} Let $(X, \hat \omega)$ be an $n$-dimensional complete K\"ahler manifold with finite volume. Let $\alpha_{1}, \cdots, \alpha_n$ be smooth $d$-closed semi-positive $(1,1)$-forms  such that $\alpha_{j}\leq \hat \omega$ on $X$ for every $1\leq j\leq n$. Assume that $n\geq 2$. Put
$$
T:=\alpha_{3} \wedge \cdots \wedge \alpha_n,\ T:=1, \ \text{if}\ n=2.
$$
Then
$$
\left(\int_X \alpha_{1}\wedge \alpha_2\wedge T \right)^2 \geq \left(\int_{X} \alpha_{1}^2\wedge T\right)\left(\int_X \alpha_{2}^2\wedge T\right).
$$
\end{theorem}

\textbf{Remark}: In case $(X, \hat \omega)$ is compact K\"ahler, the above theorem is just the Khovanskii-Teissier inequality. The classical Alexandrov-Fenchel inequality follows from the above theorem with $X=\mathbb R^n\times(\mathbb R/\mathbb Z)^n$, see \cite{Wang-AF} for the proof.

\begin{proof}[Proof of Theorem \ref{th:T-AF}] We shall follow the method in \cite{Wang-AF}. Consider $\alpha_j+\epsilon \hat \omega $ instead of $\omega$, one may assume that 
\begin{equation}\label{eq:equ-norm}
\frac{\hat \omega}{C} \leq \alpha_j \leq C \hat \omega,  
\end{equation}
for every $1\leq j\leq n$, where $C$ is a fixed positive constant.

\medskip

\emph{Step 1}: By Proposition 2.6 in \cite{Wang-AF}, it suffices to show that
$$
\psi: t\mapsto -\log\int_{X} \omega_2\wedge T, \  \omega:=t\alpha_1 +(1-t) \alpha_2,
$$
is convex on $(0,1)$.

\medskip

\emph{Step 2}: Consider the trivial line bundle $\ker d:=U\times \C$ over 
$$
U:=\{t+is: 0<t<1, \ s\in \mathbb R\},
$$
with metric
$$
h(1,1)(t+is):=e^{-\psi(t)}=\int_{X} \omega_2\wedge T, 
$$
Then $\psi$ is convex iff the curvature of $(\ker d, h)$ is positive.

\medskip

\emph{Step 3}: Look at $\ker d$ as a holomorphic subbundle of 
$$
\mathcal A:= U\times A,
$$
where
$$
A:=\{f\in C^{\infty}(X, \C): \int_X |f|^2 \hat \omega_n <\infty\}.
$$
$h$ extends to a metric on $\mathcal A$ as follows:
$$
h(f,g)(t+is):=\int_{X} f\bar g ~ \omega_2\wedge T=\int_{X} \{f, * g\}, \ \forall \ f,g\in A.
$$
Thus the Chern curvature operator of $(\mathcal A, h)$ can be written as
$$
\Theta^{\mathcal A}_{tt}:=[*^{-1} (\partial/\partial t) *, \partial/\partial t].
$$
Our main theorem gives 
$$
*^{-1} (\partial/\partial t) *= \partial/\partial t +[\Lambda, \theta], \ \theta:=[\partial/\partial t, \omega]=\alpha_1-\alpha_2,
$$
thus we have
$$
\Theta^{\mathcal A}_{tt}=[[\Lambda, \theta], \partial/\partial t ]=[\theta^*, \theta],
$$
by \eqref{eq:main-1}. 

\medskip

\emph{Step 4}: Denote by $\Theta^{\mathcal K}_{tt}$ the Chern curvature operator of $(\ker d, h)$. By the subbundle curvature formula, we have
$$
\psi_{tt}= \frac{h(\Theta^{\mathcal K}_{tt} 1, 1)}{h(1,1)}= \frac{h(\Theta^{\mathcal A}_{tt} 1, 1)}{h(1,1)}- \frac{h((\Lambda\theta)_{\bot}, (\Lambda\theta)_{\bot})}{h(1,1)},
$$
where $\Lambda\theta=[\Lambda, \theta] 1$ and
$$
(\Lambda\theta)_{\bot}:= (\Lambda\theta)-\frac{h(\Lambda\theta, 1)}{h(1,1)}.
$$
is the $L^2$-minimal solution of 
$$
d(\cdot)= d(\Lambda\theta).
$$

\medskip

\emph{Step 5}: Theorem \ref{th:T-kahler} gives
$$
d(\Lambda\theta)=(d^c)^* \theta.
$$
If $u$ is a smooth one-form with compact support on $X$ then
$$
|((d^c)^* \theta, u)|^2=|(\theta, d^c u)|^2 \leq ||\theta||^2 ||d^c u||^2,
$$
moreover, theorem \ref{th:T-BKN} gives
$$
||d^c u||^2 \leq ||du||^2+||d^* u||^2,
$$
thus H\"ormander's $L^2$-theory (here we use \eqref{eq:equ-norm} and the completeness of $(X, \hat \omega)$, see the proof of Lemma 5.2 in \cite{Wang-AF} for the details) implies
$$
h((\Lambda\theta)_{\bot}, (\Lambda\theta)_{\bot})=||(\Lambda\theta)_{\bot}||^2\leq ||\theta||^2=h(\Theta^{\mathcal A}_{tt} 1, 1),
$$
where the last identity follows from $\Theta^{\mathcal A}_{tt}=[\theta^*, \theta]$. Thus $\psi_{tt}\geq 0$.  
\end{proof}

\subsection{Dinh-Nguy\^en's theorem}

Let $(X, \omega)$ be an $n$-dimensional compact K\"ahler manifold. Let $\alpha_1, \cdots, \alpha_n$ be smooth K\"ahler forms on $X$. Let $\mathcal A^{p,q}$ be the space of smooth $(p,q)$-forms on $X$ and $\mathcal A^k$ be the space of real-valued smooth $k$-forms on $X$. We have the Dolbeault cohomology group (a $\C$-vector space in fact)
$$
H^{p,q}(X, \C):= \frac{\mathcal A^{p,q} \cap \ker \dbar}{\dbar \mathcal A^{p,q-1}},
$$
and the de Rham cohomology group (an $\mathbb R$-vector space)
$$
H^k(X, \mathbb R):= \frac{\mathcal A^{k} \cap \ker d}{d \mathcal A^{k-1}}.
$$
The following theorem depends on the theory of elliptic operators, see \cite{Guillemin-notes-e}, and Theorem \ref{th:T-BKN} (when $(E, h_E)$ is trivial and $T=1$).

\begin{theorem}[Hodge-Dolbeault-de Rham theorem] Let $(X, \omega)$ be an $n$-dimensional compact K\"ahler manifold. Let $\Box_{\dbar}$ be the $\dbar$-Laplacian with respect to the $(\omega, J)$-metric. Then each $H^{p,q}(X, \C)$ is $\C$-linear isomorphic to 
$$
\mathcal H^{p,q}:= \mathcal A^{p,q} \cap\ker \Box_{\dbar}
$$ 
which is finite dimensional. 
Let $\Box_{d}$ be the $d$-Laplacian with respect to the $(\omega, J)$-metric. Then each $H^k(X, \mathbb R)$ is $\mathbb R$-linear isomorphic to 
$$
 \mathcal H^{k}:=\mathcal A^{k} \cap\ker \Box_{d}.
$$
Moreover,
$$
 \mathcal H^{k} + i  \mathcal H^{k} =\oplus_{p+q=k} \mathcal H^{p,q}.
$$
\end{theorem}

The above theorem implies that every class in $H^{p,q}(X, \C)$ has a $d$-closed representative. Fix $0 \leq m\leq n$, put
$$
T:=\alpha_{m+1} \wedge \cdots \wedge\alpha_n, \ T:=1 \ \text{if} \ m=n.
$$ 

\begin{definition} We call a class $[u]$ in  $H^{p,q}(X, \C)$, $p+q=m$, a $T$-primitive class if $[u\wedge T \wedge \omega]=0$ in  $H^{p+n-m+1,q+n-m+1}(X, \C)$.
\end{definition}

\begin{theorem}[Dinh-Nguy\^en's theorem \cite{DN06}] Assume that $[u]\in H^{p,q}(X, \C)$, $p+q=m$, is a non-zero $T$-primitive class then 
$$
\int_X (-1)^{m+\cdots+1} \mathbf u\wedge \overline{J\mathbf u} \wedge T > 0,
$$
where $\mathbf u$ is a $d$-closed representative of $[u]$, $J\mathbf u=i^{p-q} \mathbf u$. 
\end{theorem}

\begin{proof}  The case $m=0$ is trivial. Assume that $m\geq 1$. Let $\mathbf u$ be a  $d$-closed representative of $[u]$. Since $[u]$ is $T$-primitive, there exists $v\in \mathcal A^{p+n-m+1, q+n-m}$ such that
$$
\mathbf u\wedge T \wedge \omega=\dbar v.
$$
Let us look at $v$ as an element in $V_T^{p+1, q}$. Timorin's mixed hard Lefschetz theorem gives $v'\in V_T^{p,q-1}$ such that
$$
v'\wedge\omega=v.
$$
By Theorem \ref{th:elliptic}, $\Box_\partial$ are elliptic on $V_T^k$ for every $k\neq m$. Thus the elliptic operator theory (see \cite{Guillemin-notes-e}) gives
$$
v'=v'_h +\Box_\partial  f', \ \ \ v'_h\in  V_T^{m-1}\cap \ker\Box_\partial , \ f'\in V_T^{m-1},
$$
The K\"ahler identity in $T$-Hodge theory implies that $\Box_\partial$ commutes with $L$ and $\Lambda$, thus
$$
v_h:=v'_h \wedge \omega\in V_T^{m+1}\cap \ker\Box_\partial, 
$$
and
$$
v=v_h +\Box_\partial f, \ \  f:=\omega\wedge f'.
$$ 
Since $\dbar v$ is $\partial$-closed, we have
$$
0=\partial\dbar v=\partial\partial^*\dbar \partial f=0.
$$
Thus 
$$
||\partial^*\dbar \partial f||^2=(\partial\partial^*\dbar \partial f, \dbar \partial f)=0,
$$
which gives $\partial^*\dbar \partial f=\dbar \partial^* \partial f=0$. The K\"ahler identity in $T$-Hodge theory implies 
$$
[\partial^*\dbar \partial, L]=0.
$$
Thus we have 
$$
\omega\wedge \dbar \partial^* \partial f'=0,
$$
which gives
$$
\dbar v=\omega \wedge \dbar v'=\omega \wedge \dbar\partial\partial^* f'.
$$
Let us write $\partial^* f'=T\wedge g$, thus
$$
(\mathbf u +\partial\dbar g)\wedge T \wedge \omega=0.
$$
Thus Timorin's mixed Hodge-Riemann bilinear relation gives
$$
\int_X (-1)^{m+\cdots+1} (\mathbf u +\partial\dbar g) \wedge \overline{J(\mathbf u +\partial\dbar g)} \wedge T\geq 0,
$$
where the equality holds iff $\mathbf u +\partial\dbar g\equiv 0$. Stokes' theorem implies
$$
\int_X (-1)^{m+\cdots+1} \mathbf u\wedge \overline{J\mathbf u} \wedge T > 0 =\int_X (-1)^{m+\cdots+1} (\mathbf u +\partial\dbar g) \wedge \overline{J(\mathbf u +\partial\dbar g)} \wedge T,
$$
thus the theorem follows. 
\end{proof}

\subsection{Curvature of higher direct images} 

We shall use the following setup:

\medskip

(1) $\pi: \mathcal X \to B$ is a proper holomorphic submersion from a complex manifold $\mathcal X$ to another complex manifold $B$, each fiber $X_t:=$ is an $n$-dimensional compact complex manifold; 

(2) $E$ is a holomorphic vector bundle over $\mathcal X$, $E_t:=E|_{X_t}$;

(3) $\omega$ is a smooth $(1,1)$-form on $\mathcal X$ that is positive on each fiber, $\omega^t:=\omega|_{X_t}$;

(4) $h_E$ is a smooth Hermitian metric on $E$, $h_{E_t}:=h_E|_{E_t}$.

\medskip

For each $t\in B$, let us denote by $\mathcal A^{p,q}(E_t)$ the space of smooth $E_t$-valued $(p,q)$-forms on $X_t$. Let us recall the following definition in \cite{BMW}:

\begin{definition} Let $V:=\{V_t\}_{t\in B}$ be a family of $\C$-vector spaces over $B$. Let $\Gamma$ be a $C^{\infty}(B)$-submodule of the space of all sections of $V$. We call $\Gamma$ a smooth quasi-vector bundle structure on $V$ if each vector of the fiber $V_t$ extends to a section in $\Gamma$ locally near $t$.
\end{definition}

Consider 
$$
\mathcal A^{p,q}:= \{\mathcal A^{p,q}(E_t)\}_{t\in B}.
$$
Denote by $\mathcal A^{p,q}(E)$ the space of smooth $E$-valued $(p,q)$-forms on $\mathcal X$. Let us define
$$
\Gamma^{p,q}:=\{u: t\mapsto u^t\in \mathcal A^{p,q}(E_t): \exists \ \mathbf u\in \mathcal A^{p,q}(E), \ \mathbf u|_{X_t}=u^t, \ \forall \ t\in B\}.
$$
We call $\mathbf u$ above a \emph{smooth representative} of $u\in \Gamma^{p,q}$. We know that each $\Gamma^{p,q}$ defines a quasi-vector bundle structure on $\mathcal A^{p,q}$.

\begin{definition}
Let $(V,\Gamma)=\oplus_{k=0}^{2n} (V^k,\Gamma^k)$ be a direct sum of quasi vector bundles over a smooth manifold $B$. Let $L$ be a section of ${\rm End}(V)$. We call $(V, \Gamma, L)$ a Lefschetz quasi vector bundle if
$$
L(\Gamma^l)\subset \Gamma^{l+2}, \ \forall \ 0\leq l\leq 2(n-1), \ L(\Gamma^{2n-1})=L(\Gamma^{2n})=0,
$$ 
and each $L^k: \Gamma^{n-k} \to \Gamma^{n+k}$, $0 \leq k\leq n$, is an isomorphism. 
\end{definition}

Same as before, one may define the Lefschetz star operator and the $sl_2$-triple on a general Lefschetz quasi vector bundle. Consider
$$
(\mathcal A, \Gamma):=\oplus_{k=0}^{2n} (\mathcal A^k, \Gamma^k), \ (\mathcal A^k, \Gamma^k):=\oplus_{p+q=k} (\mathcal A^{p,q}, \Gamma^{p,q}) 
$$
and define $L\in {\rm End}(\mathcal A)$ such that
$$
Lu(t)=\omega^t \wedge u^t, \ \forall \ u\in \Gamma.
$$
Then the hard Lefschetz theorem implies that $(\mathcal A, \Gamma, L)$ is a Lefschetz quasi vector bundle. One may also define the notion of connection on a general quasi vector bundle, see \cite{BMW}. Thus our main theorem is still true for general Lefschetz quasi vector bundles. We shall use the following connection on $(\mathcal A, \Gamma)$.

\begin{definition} The Lie-derivative connection, say $ \nabla^{\mathcal A}$, on $(\mathcal A, \Gamma)$ is defined as follows:
$$
\nabla^{\mathcal A} u:=\sum dt^j \otimes [d^E, \delta_{V_j}] \mathbf u+ \sum d\bar t^j \otimes [d^E, \delta_{\bar V_j}] \mathbf u, \ u\in \Gamma,
$$
where $d^E:=\dbar+\partial^E$ denotes the Chern connection on $(E, h_E)$ and each $V_j$ is the horizontal lift of $\partial /\partial t^j$ with respect to $\omega$.
\end{definition}

Our main theorem implies:

\begin{theorem}\label{th:star-commute} If $d\omega\equiv 0$ then the Lie-derivative connection $\nabla^{\mathcal A}$ commutes with the Lefschetz star operator $*_s$ on the Lefschetz quasi vector bundle $(\mathcal A, \Gamma, L)$.
\end{theorem}

\begin{proof} By our main theorem, it suffices to prove $[\nabla^{\mathcal A}, L]=0$, i.e. $[[d^E, \delta_{V_j}], L]=0$. Notice that $[d^E, L]=d\omega$ and
$$
 [\delta_{V_j}, L]|_{X_t}=(V_j \rfloor \omega)|_{X_t}\equiv 0,
$$
since each $V_j$ is horizontal. By the Jacobi identity, $d\omega=0$ gives $[\nabla^{\mathcal A}, L]=0$.
\end{proof}

For bidegree reason, the connection, say $D^{\mathcal A}$, on each $(V^{p,q}, \Gamma^{p,q})$ induced by $\nabla^{\mathcal A}$ satisfies
$$
D^{\mathcal A} u:=\sum dt^j \otimes [\partial^E, \delta_{V_j}] \mathbf u+ \sum d\bar t^j \otimes [\dbar, \delta_{\bar V_j}] \mathbf u, \ u\in \Gamma^{p,q}.
$$
Thus
$$
(\nabla^{\mathcal A}-D^{\mathcal A}) u=\sum dt^j \otimes [\dbar, \delta_{V_j}] \mathbf u+ \sum d\bar t^j \otimes [\partial^E, \delta_{\bar V_j}] \mathbf u.
$$
Put
$$
\kappa_j u:= [\dbar, \delta_{V_j}] \mathbf u, \ \kappa_{\bar j} u:= [\partial^E, \delta_{\bar V_j}] \mathbf u.
$$
Then 
$$
\nabla^{\mathcal A}-D^{\mathcal A}=\sum dt^j \otimes \kappa_j + \sum d\bar t^j \otimes \kappa_{\bar j}.
$$

\begin{definition} We call 
$$
\kappa_j u:= [\dbar, \delta_{V_j}] \mathbf u, 
$$
the non-cohomological Kodaira-Spencer action of $\kappa_j$ on $u\in \Gamma$.
\end{definition}

The fiberwise Hodge star operator $*$ equals $*_s\circ J$, recall that $J$ is the Weil-operator 
$$
Ju=i^{p-q} u, \ u\in \Gamma^{p,q}.
$$
Thus Theorem \ref{th:star-commute} implies:

\begin{proposition} If $d\omega\equiv 0$ then  $D^{\mathcal A}*=*D^{\mathcal A}, \kappa_j *=-*\kappa_j$ and $\kappa_{\bar j} *=-*\kappa_{\bar j}$.
\end{proposition}

\begin{theorem}\label{th:main-C} If $d\omega\equiv 0$ then $D^{\mathcal A}$ defines a Chern connection on each $(\mathcal A^{p,q}, \Gamma^{p,q})$ and each $\kappa_{\bar j}$ is the adjoint of $\kappa_j$.
\end{theorem}

\begin{proof} \emph{First part}: Since the metric on $\mathcal A^{p,q}$ is defined by
$$
(u, v)=\int_{X_t} \{u, * v\},
$$
thus the above proposition implies that $D^{\mathcal A}$ preserves the metric. The fact that the square of the $(0,1)$-part of $D^{\mathcal A}$ vanishes follows from the usual Lie derivative identity, see \cite{BMW} for the details. Thus $D^{\mathcal A}$ is a Chern connection.

\medskip

\emph{Second part}: Let $u\in\Gamma^{p,q}$, $v\in \Gamma^{p-1,q+1}$, then for bidegree reason, we have
$$
0=\partial/\partial t^j (u, v)=(\kappa_j u,v)+ (u, *^{-1} \kappa_{\bar j} * v),
$$
which gives $(\kappa_j u,v)= (u,  \kappa_{\bar j} v)$ by the above proposition.
\end{proof}

\textbf{Remark}: One may also prove the above theorem by a direct computation without using the Hodge star operator, see \cite{Naumann}. For other related results on the Lie-derivative connection, see \cite{Bern11}, \cite{GS15}, \cite{LiuYang13}, \cite{Maitani84}, \cite{MT07}, \cite{MT08}, \cite{Sch12}, \cite{Sch14}, \cite{TY15}.

\medskip

The curvature of the Lie-derivative connection is 
$$
(\nabla^{\mathcal A})^2 =\sum [ [d^E, \delta_{V_j}],  [d^E, \delta_{\bar V_k}]] dt^j \wedge d\bar t^k.
$$
For bidegree reason, we have
\begin{equation}\label{eq:curvature-C}
(D^{\mathcal A})^2= (\nabla^{\mathcal A})^2- \sum [\kappa_j, \kappa_{\bar k}]  dt^j \wedge d\bar t^k.
\end{equation}
One may get a curvature formula of the higher direct image bundles using the above formula and the sub-bundle-quotient-bundle curvature formula, see \cite{BMW} for the details.

\begin{definition} We call $u\in \Gamma^{p,q}$ a holomorphic section of $\mathcal A^{p,q}$ if each
$$
[\dbar, \delta_{\bar V_j}] \mathbf u =0,
$$
on fibers, i.e. the $(0,1)$-part of $D^{\mathcal A}$ vanishes on $u$.
\end{definition}

Theorem \ref{th:main-C} and \eqref{eq:curvature-C} give:

\begin{proposition} If $d\omega=0$ and $u$ is a holomorphic section of $\mathcal A^{p,q}$ then 
$$
i\partial\dbar ||u||^2 \geq -i ((\nabla^{\mathcal A})^2 u, u)+ i\sum ((\kappa_k^*u, \kappa_j^*u)-(\kappa_j u, \kappa_k u)) dt^j  \wedge d\bar t^k.
$$
\end{proposition}

We shall show how to use the above proposition in a future publication \cite{Wang-notes}.

\end{document}